\documentclass[a4paper,reqno]{amsart}

\usepackage{amssymb,amsmath,mathtools,mathrsfs}
\usepackage[foot]{amsaddr}
\usepackage{hyperref}
\hypersetup{
	colorlinks=true,
	linkcolor=blue,
	filecolor=blue,      
	urlcolor=blue,
	citecolor=blue,
}

\DeclareMathOperator{\Tr}{Tr}
\DeclareMathOperator{\Supp}{Supp}

\newcommand {\N}{\mathbb{N}}
\newcommand {\R}{\mathbb{R}}

\newtheorem{theorem}{Theorem}[section]
\newtheorem{lemma}[theorem]{Lemma}

\theoremstyle{definition}

\newtheorem{prop}[theorem]{Proposition}

\theoremstyle{remark}
\newtheorem{remark}[theorem]{Remark}

\numberwithin{equation}{section}

\begin{document}

\title[BKM with frequency and temporal localization]{A Beale-Kato-Majda criterion with optimal frequency and temporal localization}

\author{Xiaoyutao Luo }
\address{Department of Mathematics, Statistics and Computer Science,
	University of Illinois At Chicago}
\curraddr{Department of Mathematics, Statistics and Computer Science, University of Illinois At Chicago, Chicago, Illinois 60607}
\email{xluo24@uic.edu}

%    General info
%\subjclass[2000]{76D03, 35Q35}

%\date{\today}
\begin{abstract}
We obtain a Beale-Kato-Majda-type criterion with optimal frequency and temporal localization for the 3D Navier-Stokes equations. Compared to previous results our condition only requires the control of Fourier modes below a critical frequency, whose value is explicit in terms of time scales. As applications it yields a strongly frequency-localized condition for regularity in the space $B^{-1}_{\infty,\infty}$ and also a lower bound on the decaying rate of $L^p$ norms $2\leq p <3$ for possible blowup solutions. The proof relies on new estimates for the  cutoff dissipation and energy at small time scales which might be of independent interest.
\end{abstract}
\maketitle
\section{introduction}
Consider the three-dimensional incompressible Navier-Stokes equations (3D NSE)
\begin{align}\label{eq:3dNSE}
& \partial_t u+ (u\cdot \nabla )u -\nu \Delta u = -\nabla p \nonumber \\
& \nabla \cdot u =0  \\
& u(x,0) = u_0(x) \nonumber
\end{align}
where $u(x,t)$ is the unknown velocity, $p(x,t)$ is the scalar pressure. In this note we normalize the kinematic viscosity $\nu >0$ to $1$ and consider spatial domain $\Omega = \R^3$ or $ \mathbb{T}^3$.

In \cite{LERAY} Leray constructed global weak solutions for initial datum with finite energy. Such solutions obey the energy inequality $\|u(t) \|_2 + 2\nu \int_{t_0}^{t} \|\nabla u \|_2^2 \leq \|u(t_0)\|_2$ for all $t \in [0,T]$ and a.e $t_0 \in [0,t]$ including $0$. Although the uniqueness and global regularity of such solutions remain open questions. The classical result of Ladyzhenskaya-Prodi-Serrin states that if a Leray-Hopf weak solution lies in $L^q_tL^p_x$ for $\frac{2}{q} + \frac{3}{p} \leq 1$ with $p>3$, then the solution is smooth and unique in the class of all Leray-Hopf weak solutions with the same initial data. Since then there have been extensive studies on refining such type of criteria, for instances \cite{BG,BV,cs-reg,CT,FJNZ}. The difficult end point case $L^\infty_t L^3_x$ was proven with the work of Escauriaza-Seregin-\v Sver\' ak \cite{ESS} and later improved to $L^\infty B^{-1+3/s}_{s,q}$ with $3 <s,q<\infty$ in \cite{GKP} by Gallagher, Koch and Planchon. 

Another borderline case when the time integrability index is $1$, which is the main focus of this paper, is the Beale-Kato-Majda criterion \cite{BKM}:
\begin{equation}\label{eq:classicalBKM}
\int_0^T \|\nabla \times u \|_\infty dt <\infty \quad \Leftrightarrow \text{u is regular on} \, [0,T].
\end{equation} 
This condition \eqref{eq:classicalBKM} was weakened by Planchon in \cite{p} to
\begin{equation}\label{eq:BKM_Planchon}
\lim_{\epsilon\to 0} \sup_q \int_{T-\epsilon}^T \| \Delta_q (\nabla \times u) \|_\infty dt <c  , 
\end{equation} 
where $\Delta_q$ is the $q$'s Littlewood-Paley projection and the constant $c>0$ is small.

Note that the above two conditions \eqref{eq:classicalBKM} and \eqref{eq:BKM_Planchon} do not require any dissipation and hence also applies to the 3D Euler equations. Therefore it is natural to investigate the effect of dissipation in such type of BKM-like criterion. In \cite{cs-reg} Cheskidov and Shvydkoy proved the following improvement over the result of Planchon:
\begin{equation}\label{eq:BKM_CS}
  \int_0^T \|   \nabla \times u_{\leq Q(t)} \|_{B^{0}_{\infty,\infty}} dt <\infty \quad \text{for some wavenumber} \, \Lambda(t) = 2^{Q(t)} ,
\end{equation} 
where the wavenumber $\Lambda(t)$ is defined so that roughly speaking above $\Lambda(t)$, i.e. when $\lambda_q \geq \Lambda(t)$ all the norms $\lambda_q^{-1} \|u_q(t) \|_\infty$ are small and hence the linear term dominates. Later on the condition was further weakened in \cite{CD} by Cheskidov and Dai requiring
\begin{equation}\label{eq:BKM_CD}
\limsup_{q\to \infty } \int_{T/2}^{T} 1_{q \leq Q(\tau)} \lambda_q \| u_q \|_\infty d \tau \leq c .
\end{equation}

At last we note that similar idea of frequency localization was also used in \cite{BG} to obtain refinement of Ladyzhenskaya-Prodi-Serrin-type criteria under an extra assumption that $u \in C_t B^{-\epsilon}_{\infty,\infty}$ for some $0< \epsilon <1$.

\subsection{Main results}

One of the main motivations of this paper is to find the optimal frequency and temporal localization in the BKM-type criteria. Throughout this paper we use the notation $I_p(T)=[T-\lambda_p^{-2},T)$ for each $p\in \N$ where $T$ is the possible time of blowup and $\lambda_p =2^p$. The localized BKM criterion states as follows. 
\begin{theorem}\label{thm:BKM}
There exist universal constants $c$ and $\delta_{BKM}$ such that if a regular solution $u$ of \eqref{eq:3dNSE} on $(0,T)$ satisfies that
\begin{equation}\label{eq:BKM}
\limsup_{p \to \infty}   \int_{T -c\lambda_p^{-2}}^T  \| \nabla u_{\leq p} \|_{\infty}  dt \leq \delta_{BKM}  ,
\end{equation}
then $u$ is regular on $(0,T]$.
\end{theorem}

We note that the importance here is that the frequency localization is explicit in terms of time, which is natural in view of the parabolic scaling of the NSE $u_\lambda (x,t) = \lambda u(\lambda x,\lambda^2 t)$.

As a consequence of this theorem we obtain regularity condition involving only a fixed number of Fourier modes around the critical frequency at each time scale, which is a strong frequency localization criterion in $B^{-1}_{\infty,\infty}$, the largest critical space for the 3D NSE.

\begin{theorem}\label{thm:ESS_B-1_infty}
There exist a universal constant $\delta_{B^{-1}_{\infty,\infty}}$ and a family of frequency localization operators $\widetilde{\Delta}_p = \sum_{|r-p| \leq b}  \Delta_{r}$ where $b$ is a fixed number with the following property: If a regular solution $u$ of \eqref{eq:3dNSE} on $(0,T)$ verifies the bound
\begin{equation}\label{eq:ESS_B-1_infty}
\limsup_{p \to \infty} \sup_{ I_p}  \|\widetilde{\Delta}_p u \|_{B^{-1}_{\infty,\infty}}   \leq \delta_{B^{-1}_{\infty,\infty}}    ,
\end{equation}
then $u$ is regular on $(0,T]$.
\end{theorem}

Another application of Theorem \ref{thm:BKM} concerns the behavior of lower order Lebesgue norms ($L^p$ for $2\leq p < 3$) in the spirit of a celebrated result by Leray on the blowup speed of the Lebesgue norms: if a regular solution $u$ blows up at time $T$ then 
\begin{equation}\label{eq:speed_Leray}
\|u(t) \|_p \gtrsim_p \frac{1}{(T-t)^\frac{p-3}{2p}}
\end{equation}
for any $3 < p \leq \infty$. This together with the result of Seregin \cite{Seregin} and shows that if a blowup occurs then  $\| u\|_p$ becomes unbounded for any $p\geq 3$. It remains an open question whether the supercritical norms $\|u(t) \|_p$ for $2< p< 3$ become unbounded when the solution blows up. And to the author's knowledge, there is no result so far concerning the behavior of $L^p$ norms for $2< p< 3$.

We are able to obtain a partial result in this direction. Noticing that when $2\leq p<3$ the right hand side of \eqref{eq:speed_Leray} goes to zero as $t\to T$-, if a solution blows up, then there is a lower bound on the decaying rate of $\|u(t) \|_p$ for $2\leq p < 3$.

\begin{theorem}\label{thm:leray_lp_norm}
If a regular solution $u$ to \eqref{eq:3dNSE} on $(0,T)$ blows up at time $T$ then 
\begin{equation}\label{eq:leray_lp_norm}
\limsup_{t \to T} \|u_{\leq q(t)} (t) \|_p  (T-t)^\frac{p-3}{2p} \gtrsim \delta_{L^p}
\end{equation}
for any $2 \leq p \leq 3$, where $q(t) = [\log_2(T-t)^{ -1/2 }_+]$ and $\delta_{L^p}$ is a universal constant depending only on $p$.
\end{theorem}

We note that when $p=3$ condition \eqref{eq:leray_lp_norm} neither follows from nor does it include the result of Escauriaza-Seregin-\v Sver\' ak \cite{ESS}. In contrast to a blowup rate for $\|u \|_p$ when $p> 3$, the equation \eqref{eq:leray_lp_norm} is a lower bound that goes to zero as $t\to T$- for Lebesgue norms $\|u \|_p$ with $2 \leq p< 3$, which is quite interesting. We conjecture that similar result should hold for the convergence rate of $\|u(t) -u(T\text{-}) \|_p$ as $t \to T$-.

\subsection{Some remarks on the main results}

\begin{remark}
It is clear that Theorem \ref{thm:BKM} improves over the classical BKM. However it is difficult to compare with the results of Planchon \cite{p}, Cheskidov-Shvydkoy \cite{cs-reg0,cs-reg} and Cheskidov-Dai \cite{CD}. One of the reasons for such difficulty is that Theorem \ref{thm:BKM} uses $L^\infty$ norm in space while those results used Besov norms in space \cite{CD,cs-reg0,cs-reg} or space-time mixed Besov space \cite{p}. All listed results can be viewed as the steps towards the following conjecture:
\begin{equation*}\label{eq:BKMgoal} 
\limsup_{p \to \infty}   \int_{T -c\lambda_p^{-2}}^T  \| \nabla \times u_{ p} \|_{\infty}  dt \lesssim 1  \quad \Leftrightarrow \quad u \, \text{is regular} .
\end{equation*}
\end{remark}

\begin{remark}
We note that Theorem \ref{thm:BKM} is a universal criterion. One can easily derive a family of Ladyzhenskaya-Prodi-Serrin-type criteria from it. For example suppose $\frac{2}{r}+ \frac{3}{s} = 1$ with $s >3$ then \eqref{eq:BKM} holds if 
$$
\limsup_{p \to \infty }   \int_{T -c\lambda_p^{-2}}^T  \|   u_{\leq p} \|_{s}^r  dt \lesssim 1.
$$
\end{remark}

\begin{remark}
Since the frequency localization in \eqref{eq:BKM} is explicit in terms of time scales, our critical wavenumber $\lambda_{q(t)} = c(T-t)^{ -1/2}  $ is $L^{2,w}$ while in the conditions \eqref{eq:BKM_CS} and \eqref{eq:BKM_CD}, the results of Cheskidov-Shvydkoy \cite{cs-reg0,cs-reg} and Cheskidov-Dai \cite{CD} respectively, the wavenumber $\Lambda(t)$ was only known to be $\Lambda\in L^1$ for Leray-Hopf solutions and no explicit formula is known so far for $\Lambda (t)$. 
\end{remark}

\begin{remark}
Previously regularity of weak solutions in $B^{-1}_{\infty,\infty}$ is known if the norm $\|u \|_{L^\infty_t B^{-1}_{\infty,\infty}}$ is small or the jump $\limsup_{t \to T}\|u(t)-u(T) \|_{B^{-1}_{\infty,\infty}} $ is small (cf. \cite{cs-reg0}). The smallness assumption in \eqref{eq:ESS_B-1_infty} is essentially only for a single mode on the associated time interval and hence is satisfied under these two smallness conditions. Moreover the norm $\| u(t) \|_{B^{-1}_{\infty,\infty}}$ are allowed to have large jump or even blow up at $t=T$ under the condition \eqref{eq:ESS_B-1_infty}.
This is in line with the form of \eqref{eq:BKM}, which in a sense tells us that if blowup occurs then on the Fourier side it must be under the curve $\lambda_{q(t)}=c(T-t)^{-\frac{1}{2}}$.
\end{remark}

\subsection{Organization of the paper}
The rest of the paper is organized as follows. In Section \ref{sec:2} we briefly introduce Littlewood-Paley theory and notations used in this paper. The idea of proving the main results is to first obtain regularity in terms of small dissipation and then to do bootstrap by a dynamical argument. These will be done in Section \ref{sec:3} and Section \ref{sec:4} respectively. With all ingredients in hand we prove the main results in Section \ref{sec:5}.

\subsection*{Acknowledgements}
The author would like to thank his advisor Alexey Cheskidov for stimulating conversations and his constant encouragement. The author was partially supported by the NSF grant DMS 1517583 through his advisor Alexey Cheskidov.

\section{Preliminaries}\label{sec:2}
\subsection{Notations}
We denote by $A \lesssim B$ an estimate of the form $A \leq CB $ with some
absolute constant $C$, and by $A \sim B$ an estimate of the form $C_1B \leq  A \leq  C_2B $ with some absolute constants $C_1$, $C_2$. We write $\| \cdot \|= \|\cdot \|_{L^p} $ for Lebesgue norms. The symbol $(\cdot,\cdot)$ stands for the $L^2$-inner product. For any $p\in \mathbb{N}$ and $t>0$ we let $\lambda_p=2^p$ be the standard dyadic number.

\subsection{Littlewood-Paley decomposition}
We briefly introduce a standard Littlewood-Paley decomposition. For a detailed background on harmonic analysis we refer to \cite{Ca}. Let $\chi : \mathbb{R}^+ \rightarrow \mathbb{R}$ be a smooth function so that $\chi(\xi) =1$ for $\xi \leq \frac{3}{4}$, and $\chi(\xi) =0$ for $\xi \geq 1$. We further define $\varphi(\xi)=\chi(\lambda_1^{-1 }\xi) -\varphi(\xi) $  and $\varphi_q(\xi) = \varphi(\lambda_q^{-1}\xi)$. For a tempered distribution vector field $u$ let us denote
\begin{align*}
u_q := \mathcal{F}^{-1}(\varphi_q)*u \quad \text{for } q>-1  \quad \text{and}\quad  u_{-1} := \mathcal{F}^{-1}(\chi)*u,
\end{align*}
where $\mathcal{F}$ is the Fourier transform. With this we have $u=\sum_{q \geq -1}u_q$ in the sense of distribution. We also use the notation $\Delta_q :=\mathcal{F}^{-1}(\varphi_q)* $ for the Littlewood-Paley projection.

Also let us finally note that the Besov space $B^{s}_{p,q}$ is the space consisting of all tempered distributions $u$ satisfying
$$
\| u\|_{B^{s}_{p,q}}:=   \Big\| \lambda_r^s \|u_r \|_p \Big\|_{l^q} < \infty.
$$

By Littlewood-Paley theorem we note that for any $s \in \mathbb{R}$
$$
\|u \|_{H^s} \sim \| u\|_{B^{s}_{2,2}}.
$$

Finally let us recall the following version of Bernstein's inequality.
\begin{lemma}
Let $u$ be a tempered distribution in $\mathbb{R}^n$, and $r \geq s \geq 1$. Then for any $q$ we have that
$$
\| u_q\|_r \lesssim \lambda_q^{n(\frac{1}{s} -\frac{1}{r})} \| u_q\|_r .
$$
\end{lemma}
\section{Energy, dissipation and flux}\label{sec:3}
In this section we develop necessary estimates on the cutoff energy, dissipation and the energy flux. For any $p\in\N$ we define the energy flux through wavenumber $p$:
$$
\Pi_{\geq p}(t) = \int_{\R^3} (u\cdot \nabla)  u \cdot \Delta_{ \geq p}(u_{\geq p}) dx .
$$
It is worth noting that in \cite{ccfs} the definition of energy flux through shell $q$ was defined differently via multiplying \eqref{eq:3dNSE} by $\Delta_{ \leq p}(u_{\leq p})$ since the goal there is to prove the energy equality. In contrast we need to study the behavior of energy and dissipation at high modes at each time scale to rule out the blowup. Here the definition is well-defined since we are on the interval of regularity.

\subsection{Cutoff Energy inequality }
The first proposition is a variation of known result on the nonlinear term of the 3D NSE and 3D Euler established in \cite{ccfs}.
\begin{prop}\label{prop:cutoff-e}
Let $u$ be a Leray-Hopf weak solution. Suppose $(T',T)$ is an interval of regularity for $u$. Then for any $t \in (T',T)$ the cutoff energy function $ \|u_{\geq p}(t) \|_2^2$ verifies the following inequality:
\begin{equation}\label{eq:cutoff_energy}
\frac{d}{dt} \|u_{\geq p}  \|_2^2 +  2 \|\nabla u_{\geq p}  \|_2^2  \leq   \big| \Pi_{\geq p}(t) \big| ,
\end{equation}
where for any $t \in (T',T)$ the term $|\Pi_{\geq p}(t)|$ verifies
\begin{equation}\label{eq:fluxabove}
\big| \Pi_{\geq p}(t) \big|  \lesssim \Big[ \sum_{r \leq p}\lambda_{r-p }^2\|u_r \|_2^2 + \sum_{r > p}  \|u_r \|_2^2 \Big]\|\nabla u_{< p} \|_\infty.
\end{equation}

\end{prop}

It is worth noting that \eqref{eq:cutoff_energy} only holds for Leray-Hopf weak solutions within the interval of regularity since otherwise one loses the cancellation property of $u$ being divergence-free. The proof for this proposition is essentially a modification of the one given in \cite{ccfs}. The novelty here is that we make full use of the natural cancellation property on the interval of regularity. For completeness we give a proof here.
\begin{proof}[Proof of Proposition \ref{prop:cutoff-e}]
The equation \eqref{eq:cutoff_energy} follows from multiplying \eqref{eq:3dNSE} by $\Delta_{\geq p} u_{\geq p}$ and integrating in space. The main part is to show \eqref{eq:fluxabove}. To this end let 
$$
\Pi_{\leq p}(t): = \int_{\R^3} (u\cdot \nabla)  u \cdot \Delta_{ \leq p}(u_{\leq p}) dx .
$$ 
Since $(T',T)$ is an interval of regularity, for any $t \in (T',T)$ by divergence-free cancellation $\int_{\R^3} (u\cdot \nabla)  u \cdot u ~dx =0$ we have
$$
\Pi_{\geq p} = -\Pi_{\leq p-1} -2 \int_{\R^3} (u\cdot \nabla)  u \cdot \Delta_{\geq p }(u_{\leq p-1}) dx .
$$
We only need to show the right hand side above obeys the right bound. 

For the second term it follows from frequency support that
\begin{align*}
\int_{\R^3} (u\cdot \nabla)  u \cdot \Delta_{\geq p }(u_{\leq p-1}) dx & \lesssim \|\Delta_{\sim p }(u \otimes u) \|_1 \| \nabla u_{\leq p-1} \|_\infty \\
&\lesssim \sum_{r \sim p}  \|u_r \|_2^2  \| \nabla u_{\leq p-1} \|_\infty.
\end{align*}
Therefore this term obeys the right bound. 

It now suffices to estimate $\Pi_{\leq p} $. Following \cite{CET,ccfs} we write
$$
\Delta_{\leq p}(u \otimes u) = r_{p}(u,u) -(u-u_{\leq p})\otimes(u-u_{\leq p}) + u_{\leq p} \otimes u_{\leq p} ,
$$
where 
$$
r_{p}(u,u)  = \int \varphi_p(y)[u(x-y)-u(x))\otimes(u(x-y)-u(x)] dy .
$$ 
We further the first term decompose as
\begin{equation*}
r_{p}(u,u) =r_{p,1}(u,u)+r_{p,2}(u,u):=\Delta_{\leq p}[r_{p}(u,u)]+\Delta_{> p}[r_{p}(u,u)].
\end{equation*}
After substituting it into the flux we find that
\begin{align*}
\Pi_{\leq p}(u) =  &\int_{\R^3} \Tr[ r_{p,1}(u,u)\cdot \nabla u_{\leq p}]dx +\Tr[ r_{p,2}(u,u)\cdot \nabla u_{\leq p}]dx\\
& -\int_{\R^3} \Tr[ (u-u_{\leq p})\otimes(u-u_{\leq p}) \cdot \nabla u_{\leq p}]dx :=\Pi_1+\Pi_2+\Pi_3.
\end{align*}
To bound these terms we first use the Minkowski inequality obtaining
$$
\| r_{p,1}(u,u)\|_1 \lesssim \int \varphi_p(y)\|u_{\leq p}(\cdot-y)-u_{\leq p}(\cdot))\|_2^2 dy.
$$
By the Mean Value Theorem it follows that
$$\\
\| r_{p,1}(u,u)\|_1 \lesssim \sum_{r \leq p}\|\nabla u_r \|_2^2   \int \varphi_p(y) |y|^2dy \lesssim  \sum_{r \leq p}    \|u_r \|_2^2 \lambda_{|p-r|}^{-2}.
$$
In a similar manner we can also estimate
$$
\| r_{p,2}(u,u)\|_1 \lesssim \int \varphi_p(y)\|u_{> p}(\cdot-y)-u_{> p}(\cdot)\|_2^2 dy \lesssim \sum_{r > p}  \|u_r \|_2^2.
$$
Collecting estimates for $r_{p,1}(u,u)$ and $r_{p,2}(u,u)$, by H\"older's inequality we have
\begin{equation*}
|\Pi_1| \leq  \| r_{p,1}(u,u)\|_1  \| \nabla u_{\leq p} \|_\infty   \lesssim  \sum_{r \leq p}    \|u_r \|_2^2 \lambda_{|p-r|}^{-2} \| \nabla u_{\leq p} \|_\infty,
\end{equation*}
\begin{equation*}
|\Pi_2| \leq  \| r_{p,2}(u,u)\|_1  \| \nabla u_{\leq p} \|_\infty   \lesssim \sum_{r > p}   \|u_r \|_2^2\| \nabla u_{\leq p}  \|_\infty 
\end{equation*}
and 
\begin{equation*}
|\Pi_3| \leq  \| u_{> p}\|_2^2  \| \nabla u_{\leq p} \|_\infty   \lesssim \sum_{r > p}   \|u_r \|_2^2\| \nabla u_{\leq p} \|_\infty .
\end{equation*}

It is now clear that \eqref{eq:fluxabove} follows from these estimates.
\end{proof}
	
\subsection{Energy and dissipation}
Suppose $(T',T)$ is an interval of regularity for $u$. We let $I_p(T) = [T-\lambda_p^{-2},T)$ and define two sequences: 
$$
E_p(T) =\sup_{I_p}\|u_{\geq p}  \|_2^2 ,
$$
and 
$$
D_p(T) =\int_{I_p} \|\nabla u_{\geq p} (s) \|_2^2 ds.
$$
These two quantities are scaling invariant.

In the remaining of this section we will derive estimates for $D_p(T)$ and $E_p(T)$ involving energy flux on small time intervals. For notational simplicity we often write $D_p$ and $E_p$ when their meanings are clear from the context.

The next lemma is to control energy and dissipation on a small time interval by the flux and the dissipation on a larger time interval. In view of partial regularity theory of the 3D NSE, if $D_p$ has decay $\lambda_p^{-1}$ then the solution should be regular.  
\begin{lemma}\label{lemma:1st_iter}
For any $0<\alpha<2$ there exists a constant $b=b(\alpha) \in \N$ with the following property. Suppose $(T',T)$ is an interval of regularity for Leray-Hopf weak solution $u$. For any $p \in \N $ sufficiently large we have
\begin{equation}\label{eq:1st_iter}
 \max\{E_p,D_p \} \leq     \lambda_{b}^{-\alpha} D_{p-b}  + \int_{I_{p-b}}\big|\Pi_{\geq p}(t) \big| dt.
\end{equation}
\end{lemma}
\begin{remark}
Note that in this lemma $b$ is not very large since the main purpose of introducing the parameter $b$ is to absorb a geometric constant from Littlewood-Paley projection and for the results of this paper we only need $\alpha>1$.
\end{remark}
\begin{proof}[Proof of Lemma \ref{lemma:1st_iter}]
Given $0 <\alpha<2$ we fix a constant $b$ so that in view of Littlewood-Paley theory the following holds:
\begin{equation}\label{eq:choice_of_b}
\lambda_{p }^{2} \|u_{\geq p}  \|_2^2  \leq \lambda_{b}^{2-\alpha}    \|\nabla u_{\geq p}  \|_2^2 .
\end{equation}

Since  $(T',T)$ is an interval of regularity, for $p$ sufficiently large we have $I_{p-b} \subset (T',T)$. We choose Lipschitz functions $\varphi_{p}(t) : (-\infty,T] \to [0,1] $ such that for any $p$ we have $\varphi_{p}(t) =1 $ on $I_p$ and $\Supp \varphi_{p} \subset I_{p-b}$. Moreover these functions also satisfy the bound:
\begin{equation}\label{eq:timecutoff_bound}
|\frac{d}{dt}\varphi_{p}  | \leq \lambda_{p-b}^2.
\end{equation}
For instance we can take $\varphi_{p}  $ to be piecewise linear functions.

To prove \eqref{eq:1st_iter} multiplying \eqref{eq:cutoff_energy} by $\varphi_{p}$ and then integrating over time yield:
\begin{equation}\label{eq:1st_iter_firsteq}
\int_{\R}   \varphi_{p}  \frac{d}{dt }\|u_{\geq p}  \|_2^2  dt + 2 \int_{\R}   \varphi_{p}\|\nabla u_{\geq p}  \|_2^2  dt \leq   \int_{I_{p-b}}|\Pi_{\geq p}| dt.
\end{equation}
Integrating by parts for the first term we see that
\begin{equation}\label{eq:1st_iter_IBP}
\int \varphi_{p} \frac{d}{dt }\|u_{\geq p}  \|_2^2 dt = \|u_{\geq p} (T\text{-})\|_2^2 - \int  \frac{d}{dt }\varphi_{p} \|u_{\geq p}  \|_2^2 dt,
\end{equation}
where $\|u_{\geq p} (T\text{-})\|_2:= \lim_{t \to T\text{-}} \|u_{\geq p} (t)\|_2$.
Using the derivative bound \eqref{eq:timecutoff_bound} for $ \varphi_{p}$ and \eqref{eq:choice_of_b} we further obtain 
\begin{align}\label{eq:1st_iter_1_4}
\int \Big| \frac{d}{dt }\varphi_{p} \Big| \|u_{\geq p}  \|_2^2 dt & \leq \int_{I_{p}} \lambda_{p-b}^{2} \|u_{\geq p}  \|_2^2  \leq\lambda_{b}^{-\alpha}  \int_{I_{p}}  \|\nabla u_{\geq p}  \|_2^2  \nonumber \\
& \leq\lambda_{b}^{-\alpha}  D_{p-b}.
\end{align}
And thus combining \eqref{eq:1st_iter_IBP}, \eqref{eq:1st_iter_1_4} and \eqref{eq:1st_iter_firsteq} we find 
$$
D_p  \leq    \lambda_{b}^{-\alpha}  D_{p-b}  + \int_{I_{p-b}}|\Pi_{\geq p}| dt.
$$

Now we show that $E_p$ also obeys this bound. To do so we first choose a sequence $t_q \in I_{q } \setminus I_{q+1}$ satisfying the following condition:
\begin{equation*}
  \| \nabla u_{\geq q} (t_q) \|_2^2   \leq  \lambda_{q }^2 D_q.
\end{equation*}
Also for any $ q >0$ let $t_q^* \in I_{q}$ be such that $\sup_{I_q}  \|u_{\geq q}  \|_2^2=   \|u_{\geq q} (t_q^*) \|_2^2$.
Now notice that for any $p$ with $p -b >0$ we have
\begin{align}\label{ep:dp}
 \|u_{\geq p}  (t_{p-b}) \|_2^2 & \leq  \lambda_{p}^{-2}  \lambda_b^{2-\alpha} \| \nabla u_{\geq p}  (t_{p-b}) \|_2^2  \nonumber\\
 & \leq \lambda_{p}^{-2}  \lambda_b^{2-\alpha}  \| \nabla u_{\geq p-b}  (t_{p-b}) \|_2^2 \nonumber \\
 & \leq  \lambda_{b}^{-\alpha}  D_{p-b}.
\end{align}

Since $t_{p-b} <t_q^* $ we can now integrate from $t_{p-b}$ to $t_q^* $ to obtain
\begin{align}\label{eq:tpa_tp}
\sup_{I_p}   \|u_{\geq p} \|_2^2 -   \|u_{\geq p}(t_{p-b}) \|_2^2 + \int_{t_{p-b}}^{t_q^*}   \|\nabla u_{\geq p} \|_2^2 \leq  \int_{I_{p-b}} |\Pi_{\geq p}| dt.
\end{align}
Putting together \eqref{ep:dp},\eqref{eq:tpa_tp} we finally obtain
$$
E_p  \leq    \lambda_{b}^{-\alpha}  D_{p-b}  + c\int_{I_{p-b}}|\Pi_{\geq p}| dt.
$$
\end{proof}

\subsection{Bound the flux}
Now we will bound the term $\int_{I_{p-b}} |\Pi_{\geq p}| dt$ under the condition \eqref{eq:BKM} of Theorem \ref{thm:BKM} in terms of $E_r$ on much larger time interval $I_r$. This together with Lemma \ref{lemma:1st_iter} will allow us to use an iterative scheme to get the desire decay for $D_p$ as $p\to \infty$.

\begin{prop}\label{prop:BKM-to-1stiteration}
Let $u$ be a Leray-Hopf weak solution and $(T',T)$ be an interval of regularity of $u$. Suppose for some constants $b\in \N$ and $\delta>0$ we have
\begin{equation}\label{eq:BKM-to-1stiteration}
\limsup_{p \to \infty}  \int_{I_{p-b}}  \|  \nabla u_{\leq p} \|_\infty  \leq \delta  ,
\end{equation}
then for any $p \in \N$ sufficiently large we have
$$
 \int_{I_{p-b}} |\Pi_{\geq p}| dt \lesssim  \delta   \sum_{r \leq p-b}E_r \lambda_{|r-p+b|}^{-2} 
$$
where the implicit constant is independent of $p$, $\delta$ and $b$.
\end{prop}
\begin{proof}
Thanks to Proposition \ref{prop:cutoff-e} we introduce the split $ \int_{I_{p-a}} |\Pi_{\geq p}| dt  \leq I_1 +I_2$ with
\begin{equation}\label{eq:BKM-I_1}
I_1 = \sup_{I_{p-b}}\sum_{ r \leq p } \lambda_{|r-p|}^{-2}\|u_r \|_2^2    \int_{I_{p-b}}   \|\nabla u_{< p} \|_\infty   dt ,
\end{equation}
and
\begin{equation}\label{eq:BKM-I_2}
I_2 \lesssim  \sup_{I_{p-b}}  \sum_{ r > p} \|u_r \|_2^2  \int_{I_{p-b}}   \|\nabla u_{< p} \|_\infty   dt.
\end{equation} 

We bound these two term respectively. By assumption \eqref{eq:BKM-to-1stiteration} there exists $p_0 \in \N$ such that for any $p \geq p_0$ we have
\begin{equation}
 \int_{I_{p-b}}  \|  \nabla u_{\leq p} \|_\infty  \leq 2\delta
\end{equation}

Therefore for any $p \geq p_0$ it follows that 
\begin{equation*}
I_1 \lesssim \delta  \sup_{I_{p-b}}\sum_{ r \leq p } \lambda_{|r-p|}^{-2}\|u_r \|_2^2. 
\end{equation*}
Since the time interval is $I_{p-b}$ and $E_p$ contains all frequencies higher than $p$ we have
\begin{align*}
I_1 & \lesssim \delta \sum_{ r \leq p } \lambda_{|r-p|}^{-2} \sup_{I_{p-b}}\|u_r \|_2^2 \\
& \lesssim  \delta    \sum_{ r \leq p-b } \lambda_{|r-p+b|}^{-2}E_r.
\end{align*}

Now for $I_2$ we similarly estimate 
\begin{align*}
I_2 & \lesssim   \sup_{I_{p-b}}  \sum_{ r > p-b} \|u_r \|_2^2  \int_{I_{p-b}}   \|\nabla u_{\leq p} \|_\infty   dt \\
& \lesssim  \delta  E_{p-b}.
\end{align*}

Thus putting together the estimates for $I_1$ and $I_2$ we have
$$
\int_{I_{p-b}}|\Pi_{\geq p}| dt \lesssim   \delta   \sum_{r \leq p-b}E_r \lambda_{|r-p+b|}^{-2} .
$$
\end{proof}

The main result of this section can be summarized as the following decay estimate on $E_p$ and $D_p$.  
\begin{lemma}\label{lemma:twoconditions-to-dissipation}
For any $ 0< \alpha<2$ there exist $\delta_\alpha >0$  and $b=b(\alpha) \in \N$ with the following property:
If a Leray-Hopf weak solution $u$ that is regular on $(T',T)$ obeys the bound
\begin{equation}\label{eq:twoconditions-to-dissipation_delta}
\limsup_{p \to \infty} \int_{I_{p-b}}  \| \nabla u_{\leq p} \|_{\infty}   \leq \delta_\alpha    ,
\end{equation}
then there exists $p_0 \in \N$ so that for any $p \geq p_0$ we have
$$
\max\{D_p,E_p\} \lesssim      \lambda_p^{-\alpha}.
$$
\end{lemma}
\begin{proof}
Given any $0 < \alpha<2$ by Lemma \ref{lemma:1st_iter} and Lemma \ref{prop:BKM-to-1stiteration} there exists $b \in \N$ and large $p_0 \in \N$ such that for any $p\geq p_0$ we have
\begin{equation}\label{eq:twoconditions-to-dissipation}
\max\{D_p,E_p\}\leq   \lambda_b^{-\alpha - \frac{2-\alpha}{2}}   D_{p-b}+ c\delta  \sum_{r \leq p-b}E_r \lambda_{|r-p+b|}^{-2}.
\end{equation}
The extra factor $ \frac{2-\alpha}{2}$ is to make room for future correction of constants.

Let $M>0$ be such that $\max\{D_p,E_p\} \leq M$ for all $p$, which is possible since $u \in L^\infty L^2 \cap L^2 H^1$. Now we use induction to finish the proof. 

We claim that for any $p \geq p_0$ it verifies that
\begin{equation}\label{eq:ESS_ind}
\max\{D_p,E_p\} \lesssim \lambda_{p- p_0}^{-\alpha}.
\end{equation}

To prove this it suffices to show for any $k\in \mathbb{N}$ we have for any $p\in \N$ with $kb \leq p \leq (k+1 )b$
\begin{equation}\label{eq:ESS_ind_descrete}
\max\{D_p,E_p\} \leq  M \lambda_{b}^{-k\alpha}.
\end{equation}

First of all if  $p_0 \leq p\leq p_0 +b$, by \eqref{eq:twoconditions-to-dissipation} we have 
\begin{equation}\label{eq:ESS_ind_delta1} 
\max\{D_p,E_p\} \leq    [\lambda_b^{-\alpha- \frac{2-\alpha}{2}} +c\delta ]M   \leq M \lambda_b^{-\alpha}
\end{equation}
provided that $\delta  \lesssim_{\alpha} 1$.

Suppose that for any $k  \leq  n $ \eqref{eq:ESS_ind_descrete} holds. Then for $k=n+1$ we need to estimate for any $(n+1)b \leq p \leq (n+2)b $ the term
\begin{align*}
 \sum_{r \leq p-b}E_r \lambda_{|r-p+b|}^{-2}
\end{align*}

We can use inductive hypothesis \eqref{eq:ESS_ind_descrete} to obtain
\begin{align*}
 \sum_{r \leq p-b}E_r \lambda_{|r-p+b|}^{-2} & \leq \sum_{k \leq n }M \lambda_{b}^{-k\alpha}  \lambda_{b}^{-2|k-n|} \\
 & \leq c_\alpha M \lambda_{b}^{-(n+1)\alpha}   
\end{align*}
where $c_\alpha$ is the constant from the geometric sum (since $b$ depends on $\alpha$ as well).

Hence when $(n+1)b \leq p \leq (n+2)b $ we have that
\begin{align*}
\max\{D_p,E_p\}  & \leq   \Big( \lambda_b^{-\alpha- \frac{2-\alpha}{2}}   + c c_\alpha \delta \Big)  M \lambda_{b}^{-(n+1)\alpha} .
\end{align*}
Choosing $\delta$(depending on $\alpha$) smaller if needed, which in view of \eqref{eq:ESS_ind_delta1} is allowable, yields
\begin{align*}
\max\{D_p,E_p\}   & \leq    M \lambda_{b}^{-(n+1)\alpha} , 
\end{align*}
for any $(n+1)b \leq p \leq (n+2)b $ and hence the induction gives the desire bound.

\end{proof}

\section{Dissipation at small time scales}\label{sec:4}
In this section we will bootstrap regularity from $\max\{D_p,E_p\} \lesssim  \lambda_p^{-\alpha}$ if $\alpha >1$. In view of the classical result of Caffarelli-Kohn-Nirenburg this should be viable since the critical scaling for regularity is $D_p \lesssim \lambda_p^{-1}$. More specifically we show that there exists a constant $\delta>0$ so that if 
\begin{equation}\label{eq:evo_l2h1}
\limsup_{p \to \infty} \lambda_p \int_{I_p}   \|\nabla u_{\geq p}(s)\|_2^2 ds \leq \delta,
\end{equation}
then the solution does not blow up at $T$. The proof of this criterion builds on the ideas developed in \cite{L}. The intuition is that if the dissipation is small on high frequencies, then energy on low frequencies are controlled. On the other hand \eqref{eq:evo_l2h1} can also be viewed as an averaged smallness condition on $\|u(t) \|_{H^1} $ for high frequencies. Consequently we get smallness on both low and high frequencies which will imply the regularity of $u$ up to time $T$.

Unlike previous sections a slightly different evolution inequality for $\|u_q \|_2$ will be used here: 
\begin{equation}\label{eq:H1_eq}
\frac{d}{dt}\|u_q  \|_2^2  +\lambda_q^2 \|u_q  \|_2^2   \lesssim    \sum_{r\leq q}\lambda_r^{\frac{5}{2}} \|u_r \|_2\sum_{|r-q|\leq 2} \|u_r \|_2 \|u_q  \|_2   + \lambda_q^\frac{5}{2}\sum_{r \geq q-2} \|u_r \|_2^2  \|u_q  \|_2
\end{equation}

The proof of this estimate follows from standard paraproduct and commutator techniques. 
\begin{theorem}\label{thm:dissip-to-reg}
Let $u$ is a Leray-Hopf weak solution. Suppose \eqref{eq:evo_l2h1} holds at time $T$. Then $\|u \|_{H^1}$ is bounded on $[T-\epsilon,T]$ for some $\epsilon>0$. In particular $u$ does not blow up at $T$.
\end{theorem}
\begin{proof}
The exact value of $\delta$ will be chosen in the end and shall be universal. The proof consists of two steps. The first step is to bound the energy on low modes and large time interval. The second step is to take advantage of the averaged form of \eqref{eq:evo_l2h1} so that we can use a continuity argument.

\textbf{Step 1:} Bounding lower modes.

Let $t_q\in I_q$ be such that $\| u_q(t_q)\|_2  = \sup_{I_q} \|u_q \|_2$, which is possible since each $u_q$ is continuous in time. We can integrate \eqref{eq:H1_eq} from $t_q$ to $T$ to find
\begin{equation}\label{eq:1st_step_int}
\sup_{I_q}\|u_q \|_2^2 \lesssim \int_{I_q}\lambda_q^2 \|u_q(t) \|_2^2 + \sup_{I_q}\|u_q \|_2( J_1+J_2)
\end{equation}
where $J_1$ and $J_2$ are
\begin{equation}
J_1 = \int_{I_q} \sum_{r\leq q}\lambda_r^{\frac{5}{2}} \|u_r \|_2\sum_{|r-q|\leq 2} \|u_r \|_2  ,
\end{equation}
and respectively
\begin{equation}
J_2 = \int_{I_q} \lambda_q^\frac{5}{2}\sum_{r \geq q-2} \|u_r \|_2^2   .
\end{equation}
Thanks to \eqref{eq:evo_l2h1} there exists $q_0$ such that for any $q \geq q_0$ 
\begin{align}\label{eq:criterion_p_0}
\lambda_q \int_{I_q} \sum_{r \geq q} \lambda_r^2 \|u_r (s)\|_2^2 ds \leq 2\delta.
\end{align}

It then follows that for $q$ sufficiently large we have $J_2 \lesssim \delta \lambda_q^{\frac{1}{2}}$. Now it suffices to estimate the term
$$
J_1 \leq  \bigg[ \int_{I_q} \big( \sum_{r\leq q  }\lambda_r^{\frac{5}{2}} \|u_r \|_2 \big)^2  \bigg]^\frac{1}{2} \bigg[ \int_{I_q}\big( \sum_{|r-q|\leq 2} \|u_r \|_2 \big)^2  \bigg]^\frac{1}{2} :=J_{11}  \times J_{12} .
$$
Using Jensen's inequality we obtain
$$
J_{11}=[ \int_{I_q} \big( \sum_{r\leq q  }\lambda_r^{\frac{5}{2}} \|u_r \|_2 \big)^2  \bigg]^\frac{1}{2} \lesssim \lambda_q^{\frac{1}{2}}\bigg[ \int_{I_q}  \sum_{r\leq q  }\lambda_r^{4} \|u_r \|_2^2 \bigg]^\frac{1}{2},
$$ 
and
$$  
J_{12} =\bigg[ \int_{I_q}\big( \sum_{|r-q|\leq 2} \|u_r \|_2 \big)^2  \bigg]^\frac{1}{2} \lesssim \bigg[ \int_{I_q}  \sum_{|r-q|\leq 2} \|u_r \|_2^2  \bigg]^\frac{1}{2}.
$$

Again for $q$ sufficiently large immediately we have
$$  
J_{12} \lesssim \delta^\frac{1}{2} \lambda_q^{-\frac{3}{2}}.
$$

For the other term $J_{11}$ we introduce the split in the same spirit as before:
\begin{align}
J_{11}^2   & = \lambda_q \int_{I_q}  \sum_{r\leq q/2  }\lambda_r^{4} \|u_r \|_2^2 +\lambda_q \int_{I_q}  \sum_{q/2 \leq r \leq q  }\lambda_r^{4} \|u_r \|_2^2 \\
&\lesssim  \lambda_q \int_{I_q}  \sum_{r\leq q/2  }\lambda_r^{4}   +  \lambda_q\sum_{q/2 \leq r \leq q  }\lambda_r^{4} \int_{I_r}\|u_r \|_2^2 \\
& \lesssim \lambda_q + \delta \lambda_q^2 \lesssim \delta \lambda_q^2 ,
\end{align}
provided that $q$ is sufficiently large.

Therefore collecting the estimates for $J_1$ and $J_2$ there exists $q_1 \in \N$ such that for any $q \geq q_1$ we have
\begin{equation}
J_1+J_2 \lesssim \delta \lambda_q^{\frac{1}{2}}.
\end{equation}
With this bound for any $q \geq q_1$  \eqref{eq:1st_step_int} becomes
\begin{equation*}
\sup_{I_q}\|u_q \|_2^2 \lesssim  \delta \big[ \lambda_q^{\frac{1}{2}} \sup_{I_q}\|u_q \|_2  +\lambda_q\big].
\end{equation*}
And finally by Young's inequality we obtain for any $q \geq q_1$
\begin{equation}\label{eq:step1_lowmodes}
\sup_{I_q}\|u_q \|_2^2 \lesssim  \delta \lambda_q .
\end{equation}

\textbf{Step 2:} A continuity argument.

We are in the position to prove the regularity. If $\delta$ is small we will show that for some sufficiently large $q$ the bound $
\sum_{r \geq q} \lambda_r^2 \|u_r  \|_2^2 \lesssim \delta \lambda_q  
$ holds on some subinterval of $I_p$ containing the endpoint $T$. 

By the Mean Value Theorem for any $q \geq q_1$ there exists $\tau_q \in (T-\lambda_q^{-2}, T)$ such that
\begin{equation}\label{keyeq:criterion_intervaltau_q}
 \sum_{r \geq q } \lambda_r^{2} \|u_r(\tau_q) \|_2^2   \leq \delta  \lambda_q   ,
\end{equation}
i.e. the desired bound is satisfied at $t= \tau_q$.

By the continuity of strong solutions there exists an nonempty interval $[\tau_q,t_q]$ so that
\begin{equation}\label{eq:criterion_interval}
\sum_{r\geq q } \lambda_r^{ 2} \|u_r \|_2^2   \leq 2\delta \lambda_q  \quad \text{for any}  \quad t \in [\tau_q,t_q] .
\end{equation}
Next we will use a continuity argument to show that if \eqref{eq:criterion_interval} holds and $t_q \leq T$, then 
$$
\sum_{r \geq q } \lambda_r^{ 2} \|u_r(t_q) \|_2^2   < 2\delta \lambda_q  .
$$
This will imply the boundedness of $\|u \|_{H^1}$ on $[\tau_q,T]$ and hence the regularity of $u$.

Consider on $I_q$ for any $p \geq q$ the equation
\begin{equation*} 
\frac{d}{dt}\|u_p(t) \|_2 +\lambda_p^2 \|u_p(t) \|_2  \lesssim    \sum_{r\leq p}\lambda_r^\frac{5}{2} \|u_r \|_2 \sum_{|r-p|\leq 2} \|u_r \|_2 + \sum_{r \geq p-2} \|u_r \|_2^2 \lambda_p^\frac{5}{2}
\end{equation*} 
 
We claim that if $q$ is sufficiently large then on $[\tau_q,t_q]$ for any $p \geq q$ we have
 \begin{equation} \label{dissipationgronwall}
\frac{d}{dt}\|u_p(t) \|_2 +\lambda_p^2 \|u_p(t) \|_2  \lesssim    \delta   \lambda_q \lambda_p^{\frac{1}{2}}  .
\end{equation} 
With this claim it is clear that by the estimate \eqref{keyeq:criterion_intervaltau_q} and the Gronwall inequality and choosing $\delta$ small enough we have
\begin{equation}
\sum_{r \geq q } \lambda_r^{ 2} \|u_r(t_q) \|_2^2   < 2\delta \lambda_q ,  
\end{equation}
provided that $t_q \leq T$, which completes the proof.

It remains to prove \eqref{dissipationgronwall}, for which it suffices to bound $  \sum_{r\leq p}\lambda_r^\frac{5}{2} \|u_r  \|_2$ and $\sum_{r \geq p-2} \|u_r \|_2^2 $ on $[\tau_q,t_q]$. For any $ t \in [\tau_q,t_q]$ we estimate
\begin{align}\label{2ndlow}
 \sum_{r\leq p}  \lambda_r^\frac{5}{2}  \|u_r(t) \|_2 & \lesssim   \sum_{r \leq q_1}\lambda_r^\frac{5}{2}    + \sum_{  q_1 < r \leq q }  \sup_{I_r}  \lambda_r^\frac{5}{2} \|u_r \|_2   + \sum_{  q  < r \leq p }  \sup_{I_r}   \lambda_r^\frac{5}{2} \|u_r \|_2 \nonumber \\
 & \lesssim  \lambda_{q_1}^\frac{1}{2}+  \delta^\frac{1}{2} \lambda_q^\frac{1}{2} \lambda_p^\frac{3}{2} \lesssim \delta^\frac{1}{2} \lambda_q^\frac{1}{2} \lambda_p^\frac{3}{2}
\end{align}
where we have used bounded energy for $r \leq q_1$, \eqref{eq:step1_lowmodes} for $q \leq r \leq p$ and \eqref{eq:criterion_interval} for $q_1 \leq r \leq q$ in the summand. The last inequality is due to $ q \gg q_1$.

We now estimate 
\begin{align}\label{2ndhigh}
\sum_{r \geq p-2} \|u_r(t) \|_2^2 \leq  \sum_{r \geq p} \|u_r(t) \|_2^2 +\sum_{p-2 \leq  r  \leq p} \|u_r(t) \|_2^2 \lesssim  \delta  \lambda_p^{-2} \lambda_q 
\end{align}
where we have used \eqref{eq:criterion_interval} for the first part and \eqref{eq:step1_lowmodes} for the second part.

Note that all implicit constants are independent of $q$ and $p$. Putting together \eqref{2ndlow} and \eqref{2ndhigh} we obtain \eqref{dissipationgronwall}.

\end{proof}

\section{Proof of main theorems}\label{sec:5}
With all ingredients in hand we prove Theorem \ref{thm:BKM}, \ref{thm:ESS_B-1_infty} and \ref{thm:leray_lp_norm}.
\begin{proof}[Proof of Theorem \ref{thm:BKM} and \ref{thm:ESS_B-1_infty}]

To prove Theorem \ref{thm:BKM} we fix some $ 1 <\alpha <2$ to apply Lemma \ref{lemma:twoconditions-to-dissipation}. Then we choose the constant $c$ in condition \eqref{eq:BKM} so that $c\lambda_p^{-2} =   \lambda_{p-b}^{-2}$ and choose the constant $\delta_{BKM} \leq \delta_\alpha$. Now all the constants in condition \eqref{eq:BKM} have been specified. Since the condition \eqref{eq:BKM} implies \eqref{eq:twoconditions-to-dissipation_delta} we have 
$$
\max\{D_p,E_p\} \lesssim      \lambda_p^{-\alpha},
$$
for any sufficiently large $p\in \N$.

Thanks to the fact that $\alpha>1$ the condition \eqref{eq:evo_l2h1} is satisfied. Hence we can apply Theorem \ref{thm:dissip-to-reg} to conclude that the norm $\|u \|_{H^1}$ is bounded on $[T-\epsilon,T]$ for some small $\epsilon$ which finishes the proof.

Theorem \ref{thm:BKM} follows from the fact that \eqref{eq:BKM} holds when \eqref{eq:ESS_B-1_infty} is satisfied. Indeed we have 
\begin{align}
  \int_{I_{p-b}}  \| \nabla u_{\leq p} \|_{\infty}  dt & \lesssim   \int_{I_{p-b}}\sum_{r \leq p}  \lambda_r \|   u_{r} \|_{\infty}  dt  \nonumber \\
  & \lesssim  \lambda_p^{-2} \sup_{I_{p-b}}\sum_{r \leq p}  \lambda_r \|   u_{r} \|_{\infty}    .
\end{align}
Define the frequency localization operator $\widetilde{\Delta}_p = \sum_{|r-p |\leq b} \Delta_r$ so that 
\begin{equation}
 \int_{I_{p-b}}  \| \nabla u_{\leq p} \|_{\infty}  dt \lesssim  \lambda_p^{-2} \sup_{I_{p-b}}\sum_{r \leq p-b}  \lambda_r \| \widetilde{\Delta}_r  u   \|_{\infty}    .
\end{equation}
Using the energy to bound low modes it follows from the above that
\begin{align}
\int_{I_{p-b}}  \| \nabla u_{\leq p} \|_{\infty}  dt 
& \lesssim  \lambda_p^{-2} \sum_{r \leq p/2} \lambda_r^{5/2} + \lambda_p^{-2} \sum_{p/2 \leq r \leq p-b }  \sup_{I_r}\lambda_r \| \widetilde{\Delta}_r  u   \|_\infty  .
\end{align}
Therefore if \eqref{eq:ESS_B-1_infty} holds, for any sufficiently large $p$ we have
\begin{align}
\int_{I_{p-b}}  \| \nabla u_{\leq p} \|_{\infty}  dt 
& \lesssim    \lambda_p^{-3/4} +\delta_{B^{-1}_{\infty,\infty}}.
\end{align}
Choosing $\delta_{B^{-1}_{\infty,\infty}}$ suitably guarantees that condition \eqref{eq:BKM} holds.
\end{proof}

\begin{proof}[Proof of Theorem \ref{thm:leray_lp_norm}   ]
Suppose the bound fails, we have for any $t$ sufficiently close to $T$ that
$$
\|u_{\leq q(t)}(t) \|_p  (T-t)^\frac{p-3}{2p} \lesssim \delta_{L^p}.
$$
Let $\widetilde{I}_p = I_{p-b}= [T-\lambda_{p-b}^2,T)$ where $b$ is the constant from the proof of Theorem \ref{thm:BKM}. Using discrete time scales we have for any $q \in \N$ sufficiently large 
$$
\sup_{t\in \widetilde{I}_p } \|u_{\leq q }(t) \|_p \lambda_q^{-1+ \frac{3}{p}}  \lesssim \sup_{t\in \widetilde{I}_p }   \|u_{\leq q(t)}(t) \|_p (T-t)^\frac{p-3}{2p} \lesssim \delta_{L^p}.
$$ 
Since $2\leq p <3$, Bernstein's inequality gives for any $q \in \N$ sufficiently large 
$$
\lambda_q^{-2}  \sup_{t\in \widetilde{I}_p }  \| \nabla u_{\leq q}(t) \|_\infty \lesssim \lambda_q^{-1+ \frac{3}{p}}   \sup_{t\in \widetilde{I}_p }  \|  u_{\leq q}(t) \|_p \lesssim \delta_{L^p},
$$
which implies that
$$
\limsup_{q\to \infty} \lambda_q^{-2} \sup_{\widetilde{I}_q}  \| \nabla u_{\leq q} \|_\infty \lesssim \delta_{L^p}.
$$
By choosing $ \delta_{L^p}$ suitably small, this in turn guarantees \eqref{eq:BKM} holds and hence the solution does not blow up at $T$, a contradiction. 
\end{proof}

%\bib{BW}{article}{
%author={Bertram, A.},
%author={Wentworth, R.},
%title={Gromov invariants for holomorphic
%maps on Riemann surfaces},
%date={1996},
%journal={jams},
%volume={9},
%number={2},
%pages={529\ndash 571},}

%

\end{document}